\documentclass[11pt,a4paper]{article}
\usepackage[T1]{fontenc}
\usepackage[utf8]{inputenc}
\usepackage{lmodern}
\usepackage{microtype}
\usepackage{amsmath,amssymb,amsthm,mathtools}
\usepackage{geometry}
\usepackage{enumitem}
\usepackage[colorlinks=true,linkcolor=black,citecolor=black,urlcolor=blue]{hyperref}
\setlist[itemize]{leftmargin=*,topsep=2pt,itemsep=2pt}
\setlist[enumerate]{leftmargin=*,topsep=2pt,itemsep=2pt}
\allowdisplaybreaks
\newtheorem{theorem}{Theorem}[section]
\newtheorem{lemma}[theorem]{Lemma}
\newtheorem{proposition}[theorem]{Proposition}

\theoremstyle{definition}
\newtheorem{definition}[theorem]{Definition}
\theoremstyle{remark}

\newcommand{\calH}{\mathcal H}
\newcommand{\calD}{\mathcal D}
\newcommand{\eps}{\varepsilon}
\newcommand{\Win}{\operatorname{Win}}

\title{A Winner Classification for Rooted Partizan Edge Geography\\on Heterogeneous Bouquets of Cycles}
\author{Adam Til\`epe}
\date{25 October 2026}
\begin{document}
\maketitle
\begin{abstract}
We consider an undirected edge-traversal game on a finite bouquet of simple
cycles. Each of two players controls a token, both tokens begin at the common
root, and the players alternate traversing edges that have not been used by
either token. A player with no legal move loses. We prove that the first
player wins precisely when the number of cycles of odd edge-length is odd,
without assuming equal cycle lengths. The proof uses a simultaneous induction
on root and defense configurations, with an exact analysis of the forced race
between two opened cycles. We also explain how the classification corresponds
to a completed Lean~4 formalization of the underlying game on finite edge sets.
\end{abstract}

\noindent\textbf{Keywords.} Combinatorial games; partizan edge geography;
undirected graph games; bouquets of cycles; parity; formal verification.

\section{Introduction}

The two tokens in the game considered here are distinct: on their respective
turns, players move only their own token. The restriction on used edges,
however, is shared. This interaction makes the winner depend on more than the
individual routes taken by the two players.

Partizan geography with separate tokens controlled by the players was
studied by Fraenkel and Simonson~\cite{fraenkel1993}. Buchanan et
al.~\cite{buchanan2026} investigate Trail Trap, a related game in which
the players choose their initial edges during play rather than starting
from prescribed token positions. Their analysis gives a winner-determination
algorithm for trees and results for certain subclasses of complete
bipartite graphs and grid graphs, while also posing further questions about
this game on graph classes. These results motivate the search for explicit
winner classifications on structured families of graphs. More recently,
Okada~\cite{okada2026} established complexity results for undirected
partizan edge geography with specified token positions.

Here we determine the winner for a different, fixed initialization: both
tokens start at the same prescribed vertex, and the graph is a bouquet of
cycles meeting only at that vertex. Thus our classification applies to a
family containing cycles, with no restriction that the cycle lengths agree;
it does not claim a solution to the initial-edge-choice version of Trail
Trap or to winner determination on arbitrary graphs. Throughout, a cycle's
length means its \emph{number of edges}.

\begin{theorem}[Rooted bouquet classification]\label{thm:main}
Let $G$ be a finite bouquet of simple undirected cycles, each having at least
three edges, and let $r$ be its common vertex. Initially both tokens are at
$r$, no edge has been used, and the first player moves first. Then the first
player has a winning strategy if and only if the number of cycles of odd
edge-length is odd.
\end{theorem}

The argument below is a direct mathematical presentation of the winner
classification established in the accompanying Lean source. Its auxiliary
configurations expose the root/defense structure underlying the more general
corridor-and-petal model used by the formalization.

\section{Game states and bouquet notation}

Let $V$ be a vertex set and $E$ a finite set of unordered pairs of distinct
vertices. A state is specified by $s=(p_0,p_1,U;t)$, where $p_i$ is the position
of token $i$, $U\subseteq E$ is the set of used edges, and $t\in\{0,1\}$ is
the player whose turn it is. A move selects a neighbor $v$ of $p_t$ for which
$\{p_t,v\}\notin U$, sends $p_t$ to $v$, adjoins the traversed edge to $U$,
and switches the turn. Both tokens are allowed to occupy the same vertex.
The player without a legal move loses.

Every legal move increases $|U|$ by one. Thus play terminates after finitely
many moves. For any state $s$, write $\Win(s)$ when the player whose turn it
is can force a win. The elementary finite-game recursion is
\begin{equation}\label{eq:bellman}
 \Win(s)\quad\Longleftrightarrow\quad
 \exists\text{ a legal move }s\longrightarrow s'\text{ with }\neg\Win(s').
\end{equation}
The equivalent two-turn formulation is the inductive predicate
\texttt{edgeGameForceWin} in the Lean development: a player selects a legal
move and must survive every legal response, including the case where no
response exists.

Let $I$ be a finite index set of cycles. For each $k\in I$, fix an ordering
of the non-root vertices around its cycle,
\begin{equation}\label{eq:cycle}
 r,z_{k,0},z_{k,1},\ldots,z_{k,m_k-1},r,
 \qquad m_k\geq 2.
\end{equation}
The listed interior vertices are pairwise distinct, avoid $r$, and are
disjoint from the interior vertices of every other cycle. These are exactly
the structural requirements of \texttt{BouquetWellFormed} in Lean. The
cycle indexed by $k$ has $m_k+1$ edges, and therefore has odd edge-length
if and only if $m_k$ is even.

For $T\subseteq I$, put
\begin{equation}\label{eq:e}
 e(T)=\#\{k\in T:m_k\text{ is even}\},\qquad
 \eps_k=\begin{cases}1,&m_k\text{ even},\\0,&m_k\text{ odd}.
 \end{cases}
\end{equation}
Consequently $e(T\setminus\{k\})=e(T)-\eps_k$ whenever $k\in T$.

\section{Two elementary families of positions}

Our proof uses positions in which the unused edges form a bouquet, together
with a single partially traversed cycle. Completely exhausted cycles can be
deleted from the graph without affecting the legal continuation of play.
Interchanging the two player labels likewise preserves the identity of the
winner relative to the player whose turn it is.

\begin{definition}[Root and defense positions]\label{def:positions}
For a finite set of petals $T$, let $\calH(T)$ be a position in which both
tokens are at $r$, no edge of any petal in $T$ is used, and one player is
designated to move. We also write $\calH(T)$ for the assertion that the
player to move wins.

For $k\in T$ and $0\leq q<m_k$, let $\calD(T;k,q)$ be the position in which
the player to move has a token at $r$, the other token is at $z_{k,q}$, and
the used edges are exactly
\begin{equation}\label{eq:usedD}
 \bigl\{\{r,z_{k,0}\}\bigr\}
 \ \cup\ \bigl\{\{z_{k,i},z_{k,i+1}\}:0\leq i<q\bigr\}.
\end{equation}
All other edges of the bouquet on $T$, including the opposite root edge
$\{r,z_{k,m_k-1}\}$, remain unused. Write $\calD(T;k,q)$ for the assertion
that the player at $r$ can force a win.
\end{definition}

The defense coordinate is always measured from the consumed entrance in the
chosen orientation. In particular, the two choices of initial entrance
into a fresh cycle are isomorphic by reversing that cycle and fixing all
other cycles.

\begin{lemma}[First-opening equation]\label{lem:opening}
For every finite $T$,
\begin{equation}\label{eq:opening}
 \calH(T)\quad\Longleftrightarrow\quad
 \exists k\in T:\neg\calD(T;k,0).
\end{equation}
\end{lemma}
\begin{proof}
From the double-root position, any legal move must take the moving token
from $r$ into one of the cycles. The opponent remains at $r$ and has the
next turn. After choosing an orientation that begins with the traversed
edge, this is precisely the defense position $\calD(T;k,0)$. The other
entrance to the same cycle gives an isomorphic state by reflection. The
claim follows from \eqref{eq:bellman}.
\end{proof}

\section{The forced race between two petals}

Fix $\calD(T;k,q)$ and define
\begin{equation}\label{eq:A}
 A=m_k-1-q\geq0.
\end{equation}
This is the number of unused internal edges before the token in petal $k$
reaches its last interior vertex.

\begin{lemma}[The opposite entrance]\label{lem:opposite}
In $\calD(T;k,q)$, entering $k$ by its opposite root edge is a winning move
for the player at the root if and only if $A$ is even.
\end{lemma}
\begin{proof}
After this move both root-incident edges of $k$ are consumed. The two tokens
are trapped at opposite ends of the remaining internal path of $A$ unused
edges, and the token originally inside $k$ moves next. The players
alternately use edges from the two ends of this path. When $A$ is even, the
inside player is the first unable to move; when $A$ is odd, the root player
is the first unable to move. The case $A=0$ is included: the inside player
has no response at all.
\end{proof}

The root player may instead open a previously untouched petal $j\neq k$.
Choose its orientation so that the new token enters $z_{j,0}$ and put
$B=m_j-1$. Denote by $P$ the player already inside $k$, and by $Q$ the
player who has just entered $j$. The next turn belongs to $P$.

\begin{lemma}[Four outcomes of the two-petal race]\label{lem:race}
The move by $Q$ into $j$ is winning for $Q$ exactly when the predicate in
the corresponding row holds:
\begin{equation}\label{eq:four}
\begin{array}{r@{\qquad}l}
 A<B: & \neg\calD(T\setminus\{k\};j,A+1),\\[2pt]
 B+2\leq A: & \calD(T\setminus\{j\};k,q+B+2),\\[2pt]
 A=B: & \neg\calH(T\setminus\{j,k\}),\\[2pt]
 A=B+1: & \calH(T\setminus\{j,k\}).
\end{array}
\end{equation}
These four cases are mutually exclusive and exhaustive.
\end{lemma}
\begin{proof}
While a token is inside its respective cycle, its incoming edge is already
used, so its next move is forced. The two forward paths are edge-disjoint.
After $n=\min(A,B)$ pairs of forced moves, the players are at interior
coordinates $q+n$ in $k$ and $n$ in $j$, and $P$ is to move.

If $A<B$, the next two moves are $P\to r$ and then one further interior
step by $Q$. The first cycle is now completely exhausted. It is $P$'s turn
at the root, in $\calD(T\setminus\{k\};j,A+1)$. Hence $Q$ wins exactly
when this defense position is losing.

If $B+2\leq A$, after the $B$ pairs of moves, $P$ advances one internal
step, $Q$ returns to the root, and $P$ advances one further internal step.
It is now $Q$'s turn at the root in
$\calD(T\setminus\{j\};k,q+B+2)$. Thus $Q$ wins precisely when that
defense position is winning.

If $A=B$, both tokens reach their respective last interior vertices at
the same time. The next two moves return first $P$ and then $Q$ to the
root, so the next player is $P$. Both petals are exhausted; the resulting
position is $\calH(T\setminus\{j,k\})$. This must be losing for $P$ if
$Q$ is to win.

Finally, if $A=B+1$, the remaining moves are an internal step by $P$,
the return of $Q$ to the root, and the return of $P$ to the root. The next
player is $Q$ in $\calH(T\setminus\{j,k\})$, which must therefore be
winning. The inequalities also ensure that the two defense coordinates
appearing above lie in their required ranges.
\end{proof}

\begin{proposition}[Defense recurrence]\label{prop:defense-recurrence}
With $A=m_k-1-q$, the position $\calD(T;k,q)$ is winning precisely when
$A$ is even or some $j\in T\setminus\{k\}$ gives a winning outcome in the
corresponding row of \eqref{eq:four}.
\end{proposition}
\begin{proof}
The only unused edges leaving the root are the opposite entrance of $k$
and the two entrances to each untouched petal $j\neq k$. The first case is
classified by Lemma~\ref{lem:opposite}, the others by
Lemma~\ref{lem:race} and reflection. Equation~\eqref{eq:bellman} completes
the argument.
\end{proof}

\section{Exact classification and proof of the main theorem}

\begin{theorem}[Root and defense formulas]\label{thm:HD}
For every finite $T$, every $k\in T$, and every $0\leq q<m_k$, put
$A=m_k-1-q$. Then
\begin{equation}\label{eq:defense-formula}
\boxed{\ \calD(T;k,q)\ \Longleftrightarrow\
 (A\text{ even})\ \lor\ (e(T\setminus\{k\})\text{ odd})\ \lor\
 \bigl(\exists j\in T\setminus\{k\}:m_j\text{ even},\ m_j<A\bigr).\ }
\end{equation}
Moreover,
\begin{equation}\label{eq:root-formula}
\boxed{\ \calH(T)\ \Longleftrightarrow\ e(T)\text{ is odd}.\ }
\end{equation}
\end{theorem}

\begin{proof}
We use simultaneous strong induction on $|T|$. For $T=\varnothing$ there
is no legal move from the double root and no defense position to consider.
Suppose that both assertions hold for every proper subset of $T$.

We first prove the defense formula for $T$ and any $k\in T$. By
Lemma~\ref{lem:opposite}, the defense is winning whenever $A$ is even.
Assume henceforth that $A$ is odd, and write
\[
 S=T\setminus\{k\},\qquad e_0=e(S).
\]
We claim that opening another petal is winning exactly if
\begin{equation}\label{eq:criterion}
 e_0\text{ is odd}
 \quad\lor\quad
 \exists j\in S:\ m_j\text{ is even and }m_j<A.
\end{equation}

\emph{Necessity.} Suppose an opening into $j\in S$ wins, and let
$B=m_j-1$. We apply Lemma~\ref{lem:race} and the induction hypothesis.

If $A<B$, the child defense position on $S$ has remaining internal gap
$h=m_j-A-2$. It must be losing. If $m_j$ were odd, then $h$ would be even
(since $A$ is odd), and the child would be winning by the first clause
of \eqref{eq:defense-formula}. Thus $m_j$ is even. The losing child must
also have an even number of other odd-length petals, namely
$e(S\setminus\{j\})=e_0-1$; consequently $e_0$ is odd.

If $B+2\leq A$ and $m_j$ is even, then
$m_j+1=B+2\leq A$, giving $m_j<A$. If instead $m_j$ is odd, the winning
child defense position has gap $h=A-m_j-1$, which is odd. Its number of
other odd-length petals is $e_0$. By induction, it can win only if $e_0$
is odd or some even $m_\ell$ with $\ell\in S\setminus\{j\}$ satisfies
$m_\ell<h<A$. Either possibility implies \eqref{eq:criterion}.

If $A=B$, we have $m_j=A+1$ even. The winning condition in the third row
is that $\calH(S\setminus\{j\})$ loses. Since
$e(S\setminus\{j\})=e_0-1$, the induction hypothesis forces $e_0$ to
be odd. If $A=B+1$, then $m_j=A$ is odd, and the winning root child has
$e(S\setminus\{j\})=e_0$ odd by induction. This proves necessity.

\emph{Sufficiency.} First suppose there is some $j\in S$ with $m_j$
even and $m_j<A$. Then $B+2=m_j+1\leq A$. The second row of
\eqref{eq:four} applies. Its defense child has remaining gap
$h=A-m_j-1$, which is even; hence this child is winning by induction, and
opening $j$ succeeds.

Otherwise suppose that $e_0$ is odd. Choose among the petals of $S$ with
even interior length one of minimum interior length, say $j$. If
$m_j<A$, we are in the preceding case. Since $A$ is odd and $m_j$ is
even, the only other possibilities are $m_j=A+1$ or $m_j\geq A+3$.
For $m_j=A+1$, the third row of \eqref{eq:four} applies, and the
root child has $e_0-1$ odd-length petals, an even number; it loses by
induction. For $m_j\geq A+3$, the first row applies. Its defense child
has gap $h=m_j-A-2$, which is odd, and it has $e_0-1$ other odd-length
petals, also an even number. By minimality of $m_j$, all those other
odd-length petals have interior length at least $m_j>h$. Therefore all
three disjuncts of the induction formula are false for the child. It
loses, so opening $j$ succeeds. This proves \eqref{eq:criterion}, and
therefore \eqref{eq:defense-formula} for $T$.

We now establish \eqref{eq:root-formula} for $T$ using
Lemma~\ref{lem:opening}. Write $e=e(T)$. If $e$ is even, consider any
$k\in T$. If $m_k$ is odd, the child gap $A=m_k-1$ is even and the
defense child wins. If $m_k$ is even, then
$e(T\setminus\{k\})=e-1$ is odd, and the defense child again wins.
Thus no opening is winning and $\calH(T)$ loses.

If $e$ is odd, choose a petal $k$ with even $m_k$ of minimum such
length. In $\calD(T;k,0)$ its gap $A=m_k-1$ is odd, while
$e(T\setminus\{k\})=e-1$ is even. Every other petal $j$ with even $m_j$
satisfies $m_j\geq m_k>A$, so the third clause of
\eqref{eq:defense-formula} is false as well. Therefore this defense child
loses. By Lemma~\ref{lem:opening}, $\calH(T)$ wins.
This closes the simultaneous induction.
\end{proof}

\begin{proof}[Proof of Theorem~\ref{thm:main}]
The initial game state on the bouquet is exactly the root position
$\calH(I)$. Theorem~\ref{thm:HD} therefore says that the first player
wins if and only if $e(I)$ is odd. By \eqref{eq:cycle} and
\eqref{eq:e}, $e(I)$ counts exactly the cycles of odd edge-length.
\end{proof}

\section{Correspondence with the Lean formalization}

The accompanying source file formalizes the game on an actual undirected
simple graph, rather than postulating the numerical recurrence used in the
paper proof. The object \texttt{State} records both positions, the used
undirected edges, and the turn. The predicate \texttt{Legal} and the
operation \texttt{advance} implement the legal move rule; the inductively
defined \texttt{edgeGameForceWin} implements the forced-win condition.

The construction \texttt{bouquetGraph} uses a finite union of cyclic edge
sets. The hypothesis \texttt{BouquetWellFormed} ensures precisely that the
interior lists are nondegenerate, repetition-free, disjoint from the root,
and pairwise disjoint. A petal with $m_k$ interior vertices has $m_k+1$
edges, so the source predicate \texttt{OddCycles} counts the same petals
as $e(I)$ in \eqref{eq:e}.

For the strategic classification, the formal proof develops a canonical
corridor-and-petal game. Its predicate \texttt{CanonicalFWins} refers to
a game position on that finite auxiliary graph, not to an assumed
arithmetic outcome. With $b\geq0$ the corridor budget, $M$ the remaining
petals, and $m_k$ their interior lengths, define
\begin{equation}\label{eq:F}
\begin{split}
 F(b,M)\quad\Longleftrightarrow\quad &b\equiv1\pmod2\\
 &\lor\ \#\{k\in M:m_k\equiv0\pmod2\}\equiv1\pmod2\\
 &\lor\ \exists k\in M:\ m_k\equiv0\pmod2\ \text{and }m_k+1<b.
\end{split}
\end{equation}
This is \texttt{canonicalOutcomeFormula} in the source. The formal
classification gives, for all admissible $b$ and $M$,
\begin{equation}\label{eq:canonical}
 \texttt{CanonicalFWins}(b,M,m,0)
 \quad\Longleftrightarrow\quad F(b,M).
\end{equation}
The positive-budget case is proved through the root and defense models:
\nolinkurl{twoActive_complete_four_cases} establishes the four race
boundaries; \nolinkurl{defense_model_universal} and
\nolinkurl{root_zero_parity_classification} establish the corresponding
arithmetic outcome models by induction on the number of fresh petals;
and \nolinkurl{terminal_root_classification} closes the corridor endpoint.
The result for every positive budget is
\nolinkurl{canonical_Fformula_positive_budget}. The zero-budget case,
\nolinkurl{wellformed_canonical_Fzero_formula}, follows from the exact
first-opening Bellman equation and the child classifications. Together,
these yield \nolinkurl{wellformed_Fformula_all_budgets}.

In particular, setting $b=0$ in \eqref{eq:F} removes the first and third
disjuncts:
\begin{equation}\label{eq:Fzero}
 F(0,I)\quad\Longleftrightarrow\quad e(I)\equiv1\pmod2.
\end{equation}
The reduction \nolinkurl{real_B_iff_canonical_Fzero} identifies the
forced-win predicate for the actual double-root bouquet with the
zero-budget canonical position; \texttt{fresh\_root\_all} identifies all
petals as initially unused. The final source theorem is
\begin{center}
 \texttt{TrailTrap.heterogeneous\_bouquet\_classification}.
\end{center}
It has precisely the hypothesis and conclusion of
Theorem~\ref{thm:main}, expressed using \texttt{bouquetGraph},
\texttt{freshRootState}, \texttt{edgeGameForceWin}, and \texttt{OddCycles}.
The Lean source therefore supplies a mechanically checked proof of the
main classification independently of the presentation of the induction
above. The accompanying project files fix the Lean toolchain and Mathlib
revision used to check the theorem.

\section{Concluding observation}

The outcome from the prescribed common-root position depends only on the
parity of the number of odd-length cycles, regardless of their individual
lengths. Consequently, the winner can be determined directly from the list
of cycle lengths, without exploring the game tree. This gives an explicit
classification on a family of graphs with cycles. The result concerns the
specified common-root initialization; it does not assert a classification
for arbitrary token positions or for games in which the initial edges are
chosen during play.

\end{document}